\documentclass[12pt]{amsart}
\usepackage{amssymb,latexsym}
\usepackage{enumerate}

\title{Euclidean Quadratic Forms and ADC-Forms: I}
\author{Pete L. Clark}

\thanks{Partially supported by National Science Foundation grant DMS-0701771}

\address{Department of Mathematics \\ Boyd Graduate Studies Research Center \\ University 
of Georgia \\ Athens, GA 30602-7403 \\ USA}
\email{pete@math.uga.edu}

\begin{document}

\maketitle

\newtheorem{lemma}{Lemma}
\newtheorem{prop}[lemma]{Proposition}
\newtheorem{cor}[lemma]{Corollary}
\newtheorem{thm}[lemma]{Theorem}
\newtheorem{example}[lemma]{Example}
\newtheorem{thm?}[lemma]{Theorem?}
\newtheorem{schol}[lemma]{Scholium}
\newtheorem{ques}{Question}
\newtheorem{conj}[lemma]{Conjecture}
\newtheorem*{mainthm}{Main Theorem}
\newtheorem{prob}[ques]{Problem}

\newcommand{\Z}{\mathbb{Z}}
\newcommand{\R}{\mathbb{R}}

\renewcommand{\dim}{\operatorname{dim}}
\newcommand{\LC}{\operatorname{LC}}
\renewcommand{\Z}{\mathbb{Z}}
\newcommand{\Q}{\mathbb{Q}}
\renewcommand{\R}{\mathbb{R}}
\newcommand{\C}{\mathbb{C}}
\newcommand{\F}{\mathbb{F}}
\newcommand{\N}{\mathbb{N}}
\newcommand{\ord}{\operatorname{ord}}
\newcommand{\den}{\operatorname{den}}
\newcommand{\ra}{\rightarrow}
\newcommand{\Aut}{\operatorname{Aut}}
\newcommand{\Pic}{\operatorname{Pic}}
\newcommand{\Cl}{\operatorname{Cl}}
\newcommand{\Prin}{\operatorname{Prin}}
\newcommand{\PFrac}{\operatorname{PFrac}}
\newcommand{\Inv}{\operatorname{Inv}}
\newcommand{\Div}{\operatorname{Div}}
\newcommand{\pp}{\mathfrak{p}}
\newcommand{\qq}{\mathfrak{q}}
\renewcommand{\ra}{\rightarrow}
\newcommand{\Spec}{\operatorname{Spec}}
\newcommand{\Frac}{\operatorname{Frac}}
\newcommand{\OO}{\mathcal{O}}
\newcommand{\disc}{\operatorname{disc}}
\newcommand{\SL}{\operatorname{SL}}
\newcommand{\car}{\operatorname{char}}

\noindent
We denote by $\N$ the non-negative integers (including $0$).
\\ \\
Throughout $R$ will denote a commutative, unital integral domain and $K$ its fraction field.  We write $R^{\bullet}$ for $R \setminus \{0\}$ and $\Sigma_R$ for the set of height one primes of $R$. 
\\ \\
If $M$ and $N$ are monoids (written multiplicatively, with identity element $1$), a monoid homomorphism $f: M \ra N$ is \textbf{nondegenerate} if $f(x) = 
1 \iff x = 1$.

\section*{Introduction}
\noindent
The goal of this work is to set up the foundations and begin the systematic 
arithmetic study of certain classes of quadratic forms over a fairly 
general class of integral domains.  Our work here is concentrated around 
that of two definitions, that of \textbf{Euclidean form} and \textbf{ADC form}.
\\ \\
These definitions have a classical flavor, and various special cases of them  can be found (most often implicitly) in the literature.  Our work was particularly motivated by the similarities between two classical 
theorems.  

\begin{thm}(Aubry, Davenport-Cassels)
\label{INTRO1}
Let $A = (a_{ij})$ be a symmetric $n \times n$ matrix with coefficients in $\Z$, 
and let $q(x) = \sum_{1 \leq i,j \leq n} a_{ij} x_i x_j$ be a positive 
definite integral quadratic form.  Suppose that for all $x \in \Q^n$, 
there exists $y \in \Z^n$ such that $q(x-y) < 1$.  Then if $d \in \Z$ 
is such that there exists $x \in \Q^n$ with $q(x) = d$, there exists $y \in \Z^n$ such that $q(y) = d$.
\end{thm}
\noindent
Consider
$q(x) = x_1^2 + x_2^2 + x_3^2$.  It satisfies the hypotheses of 
the theorem: approximating a vector $x \in \Q^3$ by a vector $y \in \Z^3$ 
of nearest integer entries, we get \[(x_1-y_1)^2 + (x_2 - y_2)^2 + (x_3-y_3)^2 \leq \frac{3}{4} < 1 . \]
Thus Theorem \ref{INTRO1} shows that every integer which is the sum of three \emph{rational} squares is also the sum of three \emph{integral} squares.  The Hasse-Minkowski theory makes the rational representation problem routine: $d \in \Q^{\bullet}$ is $\Q$-represented by $q$ iff it is $\R$-represented by $q$ and $\Q_p$-represented by $q$ for all primes $p$. The form
$q$ $\R$-represents the non-negative rational numbers.  For odd $p$, $q$ is smooth over $\Z_p$ and hence isotropic: it $\Q_p$-represents all rational numbers.  Finally, for $a \in \N$ there are no 
primitive $\Z_2$-adic representations of $4^a \cdot 7$, so $q$ does not $\Q_2$-adically represent $7$, whereas the other $7$ classes in 
$\Q_2^{\times}/\Q_2^{\times 2}$ are all $\Q_2$-represented by $q$.  We conclude:
\begin{cor}(Gauss-Legendre Three Squares Theorem)
\label{3SQUARESCOR}
\label{INTRO2}
An integer $n$ is a sum of three integer squares iff $n \geq 0$ and 
$n$ is \emph{not} of the form $4^a (8k+7)$.
\end{cor}
\noindent
One may similarly derive Fermat's Theorem on sums of two integer squares.   The argument \emph{does not} directly apply to sums of 
four or more squares since the hypothesis is not satisfied: if $q_n(x) = x_1^2 + \ldots + x_n^2$ and we take $x = (\frac{1}{2},\ldots,\frac{1}{2})$, the best 
we can do is to take $y$ to have all coordinates either $0$ or $1$ which gives 
$q(x-y) = \frac{n}{4}$.\footnote{On the other hand, one can easily 
deduce Lagrange's Four Squares Theorem from the Three Squares Theorem and 
Euler's Four Squares Identity.}  \\ \indent This proof of Corollary \ref{3SQUARESCOR} is essentially due to L. Aubry \cite{Aubry}, but was long forgotten until it was rediscovered by Davenport 
and Cassels in the 1960s.  They did not publish their result, but J.-P. Serre 
included it in his influential text \cite{Serre}, and it is by now quite widely known. 
\\ \\
On the other hand there are the following results.
\begin{thm}(Pfister \cite{Pfister})
\label{INTRO3}
Let $F$ be a field, $\car(F) \neq 2$, let $q(x)$ 
be a quadratic form over $F$, and view it by base extension as a quadratic 
form over the polynomial ring $F[t]$.  Suppose that for $d \in F[t]$, 
there exists $x = (x_1,\ldots,x_n) \in F(t)^n$ such that $q(x) = d$.  Then 
there exists $y = (y_1,\ldots,y_n) \in F[t]^n$ such that $q(y) = d$.
\end{thm}
\noindent
\begin{cor}(Cassels \cite{Cassels})
\label{INTRO4}
Fix $n \in \Z^{> 0}$.  A polynomial $d \in F[t]$ is a sum of squares of 
$n$ rational functions iff it is a sum of squares of $n$ polynomials.
\end{cor}
\noindent
Theorems \ref{INTRO1} and \ref{INTRO3} each concern certain quadratic forms $q$ over a domain $R$ with fraction field $K$, and the common conclusion is that for all $d \in R$, $q$ $R$-represents $d$ iff it $K$-represents $d$.  This is a natural and useful property for a quadratic form $R$ over an integral domain to have, and we call such a form an \textbf{ADC form}.
\\ \\
The relationship between the hypotheses of the Aubry-Davenport-Cassels and Cassels-Pfister theorems is not as immediate.  In the former theorem, 
the hypothesis on $q$ is reminiscent of the Euclidean algorithm.  To generalize 
this to quadratic forms over an arbitrary domain we need some notion of the size of $q(x-y)$.  We do this by introducing the notion of a  \textbf{norm function} $| \cdot |: R \ra \N$ on an integral domain.  Then we define an anisotropic quadratic form $q(x) = q(x_1,\ldots,x_n)$ over $(R,|\cdot|)$ to be \textbf{Euclidean with respect to the norm} if for all $x \in K^n$, there exists 
$y \in R^n$ such that $|q(x-y)| < 1$.  We justify this notion by carrying over the proof of the Aubry-Davenport-Cassels theorem to this context: we show that for any normed ring $(R,| \cdot |)$, a Euclidean 
quadratic form $q_{/R}$ is an ADC form.  This suggests a strategy of proof 
of the Cassels-Pfister theorem: first, find a natural norm on 
the domain $R = F[t]$, and second show that any ``constant'' quadratic 
form over $R$ is Euclidean with respect to this norm.  This strategy is carried 
out in Section 2.5; in fact we get a somewhat more general (but still known) result.
\\ \\
After establishing that every Euclidean form is an ADC form, a natural followup is to identify all Euclidean forms and ADC forms over normed rings 
of arithmetic interest, especially complete discrete valuation rings (CDVRs) and 
Hasse domains: i.e., $S$-integer rings in global fields.  This is a substantial project that is begun but not completed here.  In fact much of this paper 
is foundational: we do enough work to convince the reader (or so I hope) that 
Euclidean and ADC forms lead not just to a generalization of parts of the arithmetic theory of quadratic forms to a larger class of rings, but that these 
notions are interesting and useful even (especially?) when applied to the most 
classical cases.
\\ \\
The structure of the paper is as follows: $\S 1$ lays some groundwork
regarding normed domains.  This is a topic lying at the border of commutative 
algebra and number theory, and it is not really novel: it occurs for instance 
in \cite{Lenstra} (a work with profound connections to the present subject -- 
so much so that we have chosen to leave them to a future paper), not to mention the expository work \cite{Clark-Factorization} in which the theory of factorization in integral domains is ``remade'' with norm functions playing an 
appropiately large role.  But to the best of my knowledge this theory has 
never been given a systematic exposition.  This includes the present work: 
we began with a significantly longer treatment and pared it down to include 
only those results which actually get applied to the arithmetic of quadratic 
forms.  (In particular, in an effort to convince the reader that we are doing 
number theory and not just commutative algebra, we have excised all references to Krull domains, which in fact provide a natural interpolation between UFDs and 
Dedekind domains.)  $\S 2$ introduces Euclidean quadratic forms and ADC forms and proves the main theorem advertised above: that Euclidean implies ADC.  
In $\S 3$ we prove some results on the effect of localization 
and completion on Euclideanness and the ADC property.  These results may 
not seem very exciting, but the relative straightforwardness of the proofs 
is a dividend paid by our foundational results on normed domains.  Moreover, 
they are absolutely crucial in $\S 4$ of the paper, where we completely 
dispose of Euclidean forms over a CDVR and then move to an analysis of 
Euclidean and ADC forms over Hasse domains and in particular over $\Z$ and 
$\F[t]$.  The reader who skips lightly through the rest to get to this material will be 
forgiven in advance. 
\\ \\
\textbf{Acknowlegements:} It is a pleasure to thank F. Lemmermeyer, J.P. Hanke, D. Krashen and W.C. Jagy, who each contributed valuable insights.  

\section{Normed Rings}
\subsection{Elementwise Norms} \textbf{} \\ \\ \noindent 
A \textbf{norm} on a ring $R$ is a function $|\cdot |: R \ra \N$ such that \\
(N0) $|x| = 0 \iff x = 0$, \\
(N1) $\forall x,y \in R$, $|xy| = |x||y|$, and \\
(N2) $\forall x \in R$, $|x| = 1 \iff x \in 
R^{\times}$. 
\\ \\
A \textbf{normed ring} is a pair $(R,| \cdot |)$ where $| \cdot |$ is a norm on $R$.  A ring admitting a norm is necessarily an integral domain.  We denote the fraction field by $K$. 
\\ \\
Let $R$ be a domain with fraction field $K$.  We say that two norms $| \cdot |_1, \cdot | \cdot |_2$ on $R$ are \textbf{equivalent} -- and write $| \cdot |_1 \sim | \cdot |_2$ if for all $x \in K$, $|x|_1 < 1 \iff |x|_2 < 1$.  
\\ \\
Remark: Let $(R,| \cdot |)$ be a normed domain with fraction field $K$.  By (N1) and (N2), $| \cdot |: (R^{\bullet},\cdot) \ra (\Z^+,\cdot)$ is a 
homomorphism of commutative monoids.  It therefore extends uniquely to a homomorphism on the group completions, i.e., $| \cdot |: K^{\times} \ra \Q^{>0}$ 
via $|\frac{x}{y}| = \frac{|x|}{|y|}$.  This map factors through the \textbf{group of divisibility} $G(R) = K^{\times}/R^{\times}$ to give a 
map $K^{\times}/R^{\times} \ra \Q^{>0}$, which need not be injective.
\\ \\
Example 1.1: The usual absolute value $| \cdot |_{\infty}$ on $\Z$ (inherited from $\R$) is a norm.
\\ \\
%
\noindent
Example 1.2: Let $k$ be a field, $R = k[t]$, and let $a \geq 2$ be an integer.  Then the map $f \in k[t]^{\bullet} \mapsto a^{\deg f}$ is a non-Archimedean norm $|\cdot |_a$ on $R$ and the norms obtained for various choices of $a$ are equivalent.  As we shall see, when $k$ is finite, the most natural normalization 
is $a = \# k$.  Otherwise, we may as well take $a = 2$.  
\\ \\
Example 1.3: Let $R$ be a discrete valuation ring (DVR) with valuation 
$v: K^{\times} \ra \Z$ and residue field $k$.  For any integer $a \geq 2$, we may define a norm on $R$, $| \cdot |_a: R^{\bullet} \ra \Z^{> 0}$ by $x \mapsto  a^{v(x)}$.  (Note that these are the \emph{reciprocals} of the norms $x \mapsto a^{-v(x)}$ attached to $R$ in valuation theory.)  Using the fact that 
$G(R) = K^{\times}/R^{\times} \cong (\Z,+)$ one sees that these are all 
the norms on $R$.  That is, a DVR admits a unique norm up to equivalence. 
\\ \\
Example 1.4: Let $R$ be a UFD.  Then $\Prin(R)$ is a free commutative monoid 
on the set $\Sigma_R$ of height one primes of $R$ \cite[VII.3.2]{Bourbaki}.  Thus, to give a norm map on $R$ it is necessary and sufficient to map each prime element $\pi$ to an integer $n_{\pi} \geq 2$ in such a way that if $(\pi) = (\pi')$, $n_{\pi} = n_{\pi'}$.

\subsection{Ideal norms} \textbf{} \\ \\ \noindent
For a domain $R$, let $\mathcal{I}^+(R)$ be the monoid of nonzero ideals of $R$ 
under multiplication and $\mathcal{I}(R)$ be the monoid of nonzero fractional $R$-ideals under multiplication.  
\\ \\
An \textbf{ideal norm} on $R$ is a nondegenerate homomorphism of monoids $|\cdot |: \mathcal{I}^+(R) \ra (\Z^{>0},\cdot)$.  We extend the norm to the zero ideal by 
putting $|(0)| = 0$.  In plainer language, to each nonzero ideal $I$ we assign 
a positive integer $|I|$, such that $|I| = 1 \iff I= R$ and $|IJ| = |I||J|$ 
for all ideals $I$ and $J$.

\subsection{Finite Quotient Domains} \textbf{} \\ \\ \noindent
A commutative ring $R$ has the property of \textbf{finite quotients} (FQ) if for all nonzero ideals $I$ of $R$, the ring $R/I$ is finite \cite{Butts-Wade}, \cite{Chew-Lawn}, \cite{Levitz-Mott}. \\ \indent Obviously any finite ring satisfies (FQ).  On the other hand, it can 
be shown that any infinite ring satisfying property (FQ) is necessarily a 
domain.  We define an \textbf{finite quotient domain} to be an infinite integral 
domain satisfying (FQ) which is not a field.  A finite quotient domain is a 
Noetherian domain of Krull dimension one, hence it is a Dedekind domain iff 
it is integrally closed. 
\\ \\
Example 1.5: The rings $\Z$ and $\F_p[t]$ are finite quotient domains.  From these many other examples may be derived using 
the following result.

\begin{prop}
Let $R$ be a finite quotient domain with fraction field $K$. \\
a) Let $L/K$ be a finite extension, and let $S$ be a ring with $R \subset S \subset L$.  Then, if not a field, $S$ is a finite quotient domain.  \\
b) The integral closure $\tilde{R}$ of $R$ in $K$ is a finite quotient domain.  \\
c) The completion of $R$ at a maximal ideal is a finite 
quotient domain. 
\end{prop}
\begin{proof} Part a) is \cite[Thm. 2.3]{Levitz-Mott}.  In particular, it follows from part a) that $\tilde{R}$ is a finite quotient domain.  That $\tilde{R}$ is a Dedekind ring is part of the Krull-Akizuki Theorem.  Part c) 
follows immediately from part a) and \cite[Cor. 5.3]{Chew-Lawn}.
\end{proof} 
\noindent
Let $R$ be a finite quotient domain.  For a nonzero ideal $I$ of $R$, we define 
$|I| = \# R/I$.  It is natural to ask 
whether $I \mapsto |I|$ gives an ideal norm on $R$.  

\begin{prop}
\label{2.3}
Let $I$ and $J$ be nonzero ideals of the finite quotient domain $R$.  \\
a) If $I$ and $J$ are comaximal -- i.e., $I + J = R$ -- then $|IJ| = |I||J|$.  \\
b) If $I$ is invertible, then $|IJ| = |I||J|$.  \\
c) The map $I \mapsto |I|$ is an ideal norm on $R$ iff $R$ is integrally closed.
\end{prop}
\begin{proof} Part a) follows immediately from the Chinese Remainder Theorem.  As for part b), we claim that the norm can be computed locally: for 
each $\mathfrak{p} \in \Sigma_R$, let $|I|_{\mathfrak{p}}$ be the norm of 
the ideal $I R_{\mathfrak{p}}$ in the local finite norm domain $R_{\mathfrak{p}}$.   Then 
\[|I| = \prod_{\mathfrak{p}} |I|_{\mathfrak{p}}. \]
To see this, let $I = \bigcap_{i=1}^n \mathfrak{q}_i$ 
be a primary decomposition of $I$, with $\mathfrak{p}_i = \operatorname{rad}(\mathfrak{q}_i)$.  It follows that
$\{\mathfrak{q}_1,\ldots,\mathfrak{q}_n \}$ is a finite set of pairwise 
comaximal ideals, so the Chinese Remainder Theorem applies to give 
\[R/I \cong \prod_{i=1}^n R/\mathfrak{q}_i. \]
Since $R/\mathfrak{q}_i$ is a local ring with maximal ideal corresponding to 
$\mathfrak{p}_i$, it follows that $|\mathfrak{q}_i| = |\mathfrak{q}_i R_{\mathfrak{p}_i}|$, establishing the claim.
\\ 
Using the claim reduces us to the local case, so that we may assume the ideal $I = (x R)$ is principal.  In this case the short exact sequence of $R$-modules
\[0 \ra \frac{xR}{xJ} \ra  \frac{R}{xJ} \ra \frac{R}{(x)J} \ra 0 \]
together with the isomorphism 
\[\frac{R}{J} \stackrel{\cdot x}{\ra} \frac{xR}{xJ} \]
does the job.  \\
c) If $R$ is 
integrally closed (hence Dedekind), every ideal is invertible so 
this is an ideal norm.  The converse is \cite[Thm. 2]{Butts-Wade}.
\end{proof}
\noindent
In all of our applications, $R$ is either an $S$-integer ring in a global 
field or a completion of such at a height one prime.  By the results of this 
section, the map $I \mapsto |I| = \# R/I$ is an ideal norm on these rings.  
We will call this norm \textbf{canonical}.  We ask the reader to verify 
that the norm of Example 1.1 is canonical, as are the norms $| \cdot |_{\# k}$ 
of Examples 1.2 and 1.3 when the field $k$ is finite.

\subsection{Euclidean norms}
\textbf{} \\ \\ \noindent
A norm $|\cdot |$ on $R$ is \textbf{Euclidean} if for all $x \in K$, 
there is $y \in R$ such that $|x-y| < 1$.  Whether $R$ is 
Euclidean for $| \cdot |$ depends only on the equivalence class of the norm.
\\ \\
Example 1.6: The norm $| \cdot |_{\infty}$ on $\Z$ is Euclidean.  The norms $| \cdot |_a$ on $k[t]$ are Euclidean.  For a DVR, the norms $|\cdot |_a$ (c.f. Example 4) are Euclidean: indeed, for $x \in K^{\bullet}$, $x \in K \setminus R \iff v(x) < 0 \iff |x|_a = a^{v(x)} < 1$, so we may take $y =0$.  In a similar way, to any semilocal PID $R$ one can attach a natural family of Euclidean norms (including the canonical norm if $R$ is a finite quotient domain).
\\ \\
Example 1.7: $S = \Z_K$ is the ring of integers in a number field $K$.  It is a classical problem to determine whether 
$R$ is Euclidean for the canonical norm, or \textbf{norm-Euclidean}.  Note that a norm-Euclidean number field 
has class number one.  Conditional on the Generalized Riemann Hypothesis, it is known that every number field of class number one except 
$\Q = K(\sqrt{-D})$ for $D = 19, 43, 67, 163$ is Euclidean for some non-canonical norm.\footnote{In fact the definition of a norm function one finds 
in the literature is a little weaker than ours, in that multiplicativity is 
replaced by the condition $|x| \leq |xy|$ for all $x,y \in R^{\bullet}$.}  This is to be contrasted with the standard conjecture that there are infinitely many class number one real quadratic fields and the fact that there are only finitely many \emph{norm-Euclidean} real quadratic fields \cite{BSD}.

\section{Euclidean quadratic forms and ADC forms}

\subsection{Euclidean quadratic forms} \textbf{} \\ \\
Let $(R,| \cdot |)$ be a normed ring of characteristic not $2$.  A \textbf{quadratic form} over $R$ is a polynomial $q \in R[x] = R[x_1,\ldots,x_n]$ which is homogeneous of degree $2$.  Throughout this note we only consider quadratic forms which are non-degenerate over the fraction field $K$ of $R$.  A nondegenerate quadratic form $q_{/R}$ is 
\textbf{isotropic} if there exists $a = (a_1,\ldots,a_n) \in R^n \setminus \{(0,\ldots,0)\}$ such that $q(a) = 0$; otherwise $q$ is \textbf{anisotropic}.  A form $q$ is anisotropic over $R$ iff it is anisotropic over $K$.  A quadratic 
form $q_{/R}$ is \textbf{universal} if for all $d \in R$, there exists $x \in R^n$ such that $q(x) = d$.  
\\ \\
A quadratic form $q$ on a normed ring $(R,| \cdot |)$ is \textbf{Euclidean} if for all $x \in K^n \setminus R^n$, there exists $y \in R^n$ such $0 < |q(x-y)| < 1$.  
(Again, this definition depends only on the equivalence class of the norm.)  
\\ \\
Remark: An anisotropic quadratic form $q$ is Euclidean iff for all $x \in K^n$ there exists $y \in R^n$ such that 
$|q(x-y)| < 1$.  
\begin{prop}
\label{EUCLIDRINGFORMPROP}
The norm $| \cdot |$ on $R$ is a Euclidean norm iff the quadratic form $q(x) = x^2$ is a Euclidean quadratic form.
\end{prop}
\begin{proof}
Noting that $q$ is an anisotropic quadratic form, this comes down to: \[ \forall x,y \in K, \
|x-y| < 1 \iff |q(x-y)| = |(x-y)^2| = |x-y|^2 < 1. \]
\end{proof}  
\noindent
Example 2.1: Let $n,a_1,\ldots,a_n \in \Z^+$.  Then the integral quadratic form $q(x) = a_1 x_1^2 + \ldots + a_n x_n^2$ is Euclidean iff 
$\sum_i a_i < 4$.  

\subsection{Euclideanity}
\noindent
For a quadratic form $q$ over a normed ring $(R,| \cdot |)$ with fraction field $K$, define for $x \in K^n$, 
\[E(q,x) = \inf_{y \in R^n} |q(x-y)| \]
and 
\[E(q) = \sup_{x \in K^n} E(q,x). \]
Let us call $E(q)$ the \textbf{Euclideanity} of $q$.  Thus an anisotropic form $q$ is Euclidean if $E(q) < 1$ and is not Euclidean when $E(q) > 1$.  The case $E(q) = 1$ is ambiguous: the form $q$ is \emph{not} Euclidean iff the supremum in the definition 
of $E(q)$ is attained, i.e., iff there exists $x \in K^n$ such that $E(q,x) = 1$.  A non-Euclidean form with $E(q) = 1$ will be said to be \textbf{boundary-Euclidean}.
\\ \\
We define the Euclideanity $E(R)$ of $R$ itself to be the Euclideanity of $q(x) = x^2$.  
\\ \\
Example 2.2: Take $R = \Z$ with its canonical norm and $n,a_1,\ldots,a_n \in \Z^+$, as in Example 2.1 above.  Then 
\[ E(a_1 x_1^2 + \ldots + a_n x_n^2) = \frac{ a_1 + \ldots + a_n}{4}. \]
The forms with $E(q) = 1$ are boundary-Euclidean.

\subsection{ADC-forms} \textbf{} \\ \\ \noindent
A quadratic form $q(x) = q(x_1,\ldots,x_n)$ over $R$ is an \textbf{ADC-form} if for all $d \in R$, if there exists 
$x \in K^n$ such that $q(x) = d$, then there exists $y \in R^n$ such that $q(y) = d$.  
\\ \\
Example 2.3: Any universal quadratic form is an ADC-form.  If $R = \Z$ and $q$ is positive definite and \textbf{positive universal} -- i.e., represents all positive integers -- then $q$ is an ADC-form.  Thus for each $n \geq 5$ 
there are infinitely many positive definite ADC-forms, e.g. $x_1^2 + \ldots + x_{n-1}^2 + d x_n^2$ for $d \in \Z^+$.  
\\ \\
Example 2.4: Let $\tilde{R}$ be the integral closure of $R$ in $K$. Then $q(x) = x^2$ is \emph{not} an ADC-form iff there 
exists $a \in \tilde{R} \setminus R$ such that $a^2 \in R$.  In particular $x^2$ is an ADC-form if $R$ is integrally closed.
\\ \\
Example 2.5: Let $R$ be a UFD and $a \in R^{\bullet}$.  Then $q(x) = ax^2$ is ADC iff $a$ is squarefree.
\\ \\
Example 2.6: Suppose $R$ is an algebra over a field $k$, and let $q_{/k}$ be isotropic.  Then the base extension of $q$ to $R$ is universal.  Indeed, since $q$ is isotropic over $k$, it contains the hyperbolic plane as a subform.  That is, after a $k$-linear change of variables, we may assume $q = x_1x_2 + q'(x_3,\ldots,x_n)$, and the conclusion is now clear.
\\ \\
Example 2.7: The isotropic form $q(x,y) = x^2 - y^2$ is not an ADC-form over $\Z$: it is universal over $\Q$ but not over $\Z$.    
\begin{thm}
\label{EUCIMPLIESADCTHM}
\label{MAINTHM}
Let $(R,| \cdot |)$ be a normed ring not of characteristic $2$ and $q_{/R}$ a Euclidean quadratic form.  Then $q$ is an ADC form.
\end{thm}
\begin{proof} 
For $x,y \in K^n$, 
put $x \cdot y := \frac{1}{2}(q(x+y) - q(x) - q(x))$.  Then $(x,y) \mapsto x \cdot y$ is bilinear and $x \cdot x = q(x)$.  Note that for 
$x,y \in R^n$, we need not have $x \cdot y \in R$, but certainly we have $2(x \cdot y) \in R$.    \\ \indent
Let $d \in R$, and suppose there exists $x \in K^n$ such that $q(x) = d$.  Equivalently, there exists $t \in R$ 
and $x' \in R^n$ such that $t^2 d = x' \cdot x'$.  Choose $x'$ and $t$ such that $|t|$ is minimal.  It is enough to 
show that $|t| = 1$, for then by (N1) $t \in R^{\times}$.  \\ \indent
Apply the Euclidean hypothesis with $x = \frac{x'}{t}$: there is $y \in R$ such that if $z = x - y$, 
\[ 0 < |q(z)| < 1. \]
Now put 
\[a = y \cdot y - d, \ b = 2dt - 2(x' \cdot y), \ T = at + b, \ X = ax'+by. \]
Then $a,b,T \in R$, and $X \in R^n$. \\ \noindent
{\sc Claim}: $X \cdot X = T^2 d$. \\
Indeed,
\[X \cdot X = a^2 (x' \cdot x') + ab(2 x' \cdot y) = b^2 (y \cdot y) = a^2 t^2 d + ab(2dt-b) + b^2(d+a) \]
\[= d(a^2 t^2 + 2abt + b^2) = T^2 d. \]
{\sc Claim}: $T = t(z \cdot z)$. \\
Indeed, 
\[tT = at^2 + bt = t^2 (y \cdot y) - dt^2 + 2dt^2 - t (2 x' \cdot y) \]
\[= t^2 (y \cdot y) - t(2 x' \cdot y) + x' \cdot x' = (ty-x') \cdot (ty-x') = (-tz) \cdot (-tz) = t^2 (z \cdot z). \] 
Since $0 < |z \cdot z| < 1$, we have $0 < |T| < |t|$, contradicting the minimality of 
$|t|$.
\end{proof} 
\noindent
Remark: This proof is modelled on that of \cite[pp. 46-47]{Serre}.  
\\ \\
Example 2.8: Let $R = \Z$ with its canonical norm, and consider $q_1(x,y) = x^2 + 3y^2$ and $q_2(x,y) = 2x^2 +2y^2$.  Both of these forms are non-Euclidean forms with Euclideanity $1$, i.e., boundary-Euclidean forms.  It happens that $q_1$ is nevertheless an ADC-form, a fact whose essential content was well known 
to the great number theorists of the 18th century.  For instance, one can realize $q_1$ as an index $2$-sublattice of the maximal lattice (see $\S 2.6$) 
$q'(x,y) = x^2 + xy +y^2$ which \emph{is} Euclidean (this corresponds to 
the fact that the ring of integers of $\mathbb{Q}(\sqrt{-3})$ is a Euclidean 
domain) and then reduce the problem of integer representations of $q_1$ 
to that of integer representations of $q'$ with certain parity conditions.  But in fact Weil \cite[pp. 292-295]{Weil} modifies the proof of Aubry's theorem (i.e., essentially the same argument used to prove Theorem \ref{MAINTHM}) to show directly that the boundary-Euclidean form $q_1$ is ADC.  His argument also works for the boundary-Euclidean forms $x_1^2 + x_2^2 +2 x_3^2$ and $x_1^2 + x_2^2 + x_3^2 
+ x_4^2$.  However, it does not work for $q_2$: indeed, $q_2(\frac{1}{2},\frac{1}{2}) = 1$ but $q_2$ evidently does not $\Z$-represent $1$, so $q_2$ is not ADC.  \\ \indent Is there a supplement to 
Theorem \ref{MAINTHM} giving necessary and sufficient conditions for 
a boundary-Euclidean form to be ADC?  We leave this as an open problem.

\subsection{The Generalized Cassels-Pfister Theorem} 

\begin{lemma}
\label{ANISOTROPYLEMMA}
Let $q$ be an anisotropic quadratic form over a field $k$.  Then $q$ remains anisotropic over the rational function field $k(t)$.
\end{lemma}
\begin{proof}
If there exists a nonzero vector $x \in k(t)^n$ such that $q(x) = 0$, then (since $k[t]$ is a UFD) there exists $y = (y_1,\ldots,y_n)$ such that $y \in R^n$, $\gcd(y_1,\ldots,y_n) = 1$ and $q(y) = 0$.  The polynomials 
$y_1,\ldots,y_n$ do not all vanish at $0$, so $(y_1(0),\ldots,y_n(0)) \in k^n 
\setminus (0,\ldots,0)$ is such that $q(y_1(0),\ldots,y_n(0)) = 0$, i.e., $q$ is isotropic over $k$.
\end{proof}

\begin{thm}(Generalized Cassels-Pfister Theorem)
\label{GENCPTHM}
Let $F$ be a field of characteristic not $2$, $R = F[t]$, and $K = F(t)$.  
Let $q = \sum_{i,j} a_{ij}(t) x_i x_j$ be a quadratic form over $R$.  We suppose 
that \emph{either}: \\
(i) $q$ is anisotropic and each $a_{ij}$ has degree $0$ or $1$, or \\
(ii) Each $a_{ij}$ has degree $0$, i.e., $q$ is the extension of a quadratic 
form over $k$.  \\
Then $q$ is an ADC form.
\end{thm}
\begin{proof}
Suppose first that $q$ is isotropic over $K$ and extended from a quadratic 
form $q$ over $k$.  By Lemma \ref{ANISOTROPYLEMMA}, $q_{/k}$ is isotropic.  Then by Example 2.6, $q_{/R}$ is universal.  \\ \indent
Now suppose that $q$ is anisotropic over $K$ and that each $a_{ij}$ has 
degree $0$ or $1$.  By Theorem \ref{MAINTHM}, it suffices to show that as a quadratic form over $R = k[t]$ endowed with the norm $| \cdot | = | \cdot |_2$ of Example 1.2, $q$ is Euclidean.  \\ \indent
Given an element $x = (\frac{f_1(t)}{g_1(t)},\ldots,
\frac{f_n(t)}{g_n(t)}) \in K^n$, by polynomial division we may write $\frac{f_i}{g_i} = y_i + \frac{r_i}{g_i}$ 
with $y_i,r_i \in k[t]$ and $\deg(r_i) < \deg(g_i)$.  Putting $y = (y_1,\ldots,y_n)$ and using the non-Archimedean property 
of $| \cdot |$, we find
\begin{equation}
\label{CPEQN}
|q(x-y)| = |\sum_{i,j} a_{i,j} (\frac{r_i}{g_i}) (\frac{r_j}{g_j}) | \leq \left( \max_{i,j} |a_{i,j}| \right) 
\left( \max_i  |\frac{r_i}{g_i}| \right)^2 < 1. 
\end{equation}
\end{proof}
\noindent
Remark: Example 2.5 shows that the conclusion Theorem \ref{GENCPTHM} does not extend to all forms with $\max_{i,j} \deg(a_{ij}) \leq 2$.

\subsection{Maximal Lattices} \textbf{} \\ \\ \noindent
When studying quadratic forms over integral domains it is often convenient to use the terminology 
of lattices in quadratic spaces.  Let $R$ be a domain with fraction field 
$K$, let $V$ be a finite-dimensional vector space, and let $q: V \ra K$ 
be a quadratic form.  An \textbf{R-lattice} $\Lambda$ in $V$ is a finitely 
generated $R$-submodule of $V$ such that $\Lambda \otimes_R K = V$.  A \textbf{quadratic R-lattice} is an $R$-lattice $\Lambda$ in the quadratic 
space $(V,q)$ such that $q(\Lambda) \subset R$.  
\\ \\
In particular, if $q: R^n \ra R$ is a quadratic form, then tensoring from 
$R$ to $K$ gives a quadratic form $q: K^n \ra K$ and taking $V = K^n$, $\Lambda = R^n$ 
gives a quadratic $R$-lattice.  Conversely, a quadratic lattice $\Lambda$ 
in $R^n$ which is \emph{free} as an $R$-module may be identified with a quadratic 
form over $R$.
\\ \\
A quadratic $R$-lattice $\Lambda$ is said to be \textbf{maximal} if it is 
not strictly contained in another quadratic $R$-lattice.\footnote{For the sake of brevity, we will sometimes simply say that the quadratic form $q$ is maximal if its associated free quadratic lattice is maximal.}  If $R$ is Noetherian, 
then discriminant considerations show that every quadratic $R$-lattice is 
contained in a maximal quadratic $R$-lattice.

\begin{prop}
\label{EUCIMPLIESMAXPROP}
Let $(R,|\cdot |)$ be a normed ring and $q_{/R}$ a Euclidean quadratic form.  Then 
the associated quadratic $R$-lattice $\Lambda = R^n$ is maximal.
\end{prop}
\begin{proof} For if not, there exists a strictly larger quadratic $R$-lattice 
$\Lambda'$.  Choose $x \in \Lambda' \setminus \Lambda$, so $x \in K^n \setminus R^n$.  For all $y \in \Lambda = R^n$, $x-y \in \Lambda'$, so $|q(x-y)| \in |R| = \N$.
\end{proof}
\noindent
Example 2.9: Let $(R,| \cdot |) = (\Z,| \cdot |_{\infty})$, and let $a \in \Z^{\bullet}$.  Then: \\
a) The form $ax^2$ is maximal iff it is ADC iff $a$ is squarefree.  \\
b) The form $x^2 + ay^2$ is maximal iff $a$ is squarefree and $a \equiv 1,2 \pmod 4$. 
\\ \\
Example 2.9: The form $x_1^2 + \ldots + x_n^2$ is maximal 
iff it is Euclidean iff $n \leq 3$.

\section{Localization and Completion}
\noindent
In this section we show that Euclidean forms and ADC forms behave nicely 
under localization and completion, at least if we restrict to domains $R$ 
for which norm functions (resp. ideal norm functions) have the simplest 
structure, namely UFDs (resp. Dedekind domains).

\subsection{Localization and Euclideanity}
\textbf{} \\ \\ \noindent
Suppose first that $(R,| \cdot |)$ is a normed UFD, and 
$S$ is a saturated multiplicatively closed subset.  We shall define a \textbf{localized norm} $| \cdot |_S$ on the localization $S^{-1} R$.  To do so, 
recall that $S^{-1} R$ is again a UFD and its principal prime ideals $(\pi)$ 
are precisely those for which $\pi \cap S = \varnothing$.  Therefore we may 
view the monoid $\Prin(S^{-1} R)$ as a submonoid of $\Prin(R)$ by taking it 
to be the direct sum over all the height one prime ideals $(\pi)$ of $R$ 
with $(\pi) \cap S = \varnothing$: let $\iota$ be this embedding of monoids.   We define the localized norm $| \cdot |_S: 
\Prin(S^{-1} R) \rightarrow \Z^+$ by $|x|_S := |\iota(x)|$.  
\\ \\
Remark 3.1: Here are two easy and useful properties of the localized norm:
\\ \\
$\bullet$ Any $x \in R^{\bullet}$ may be written as $s_x x'$ with $s_x \in S$ 
and $x'$ prime to $S$, and we have \[|x|_S = |s_x x'|_S = |x'|_S = |x'|. \]
$\bullet$ For any $x \in R^{\bullet}$, $|x|_S \leq |x|$.  

\begin{thm}
\label{UFDLOCTHM}
Let $(R,|\cdot |)$ be a UFD with fraction field $K$, let $S \subset R^{\bullet}$ be a saturated multiplicatively closed subset, and let $R_S$ be the localization of $R$ at $S$.  Let $q(x) \in R[x]$ be a quadratic form, and suppose that 
$E \in \R^{> 0}$ is a constant such that for all $x \in K^n$, there exists 
$y \in R^n$ such that $|q(x-y)| \leq E$.  Then for all $x \in K^n$, there 
exists $y_S \in R_S^n$ such that $|q(x-y_S)|_S \leq E$.
\end{thm}
\begin{proof}
Let $x \in K^n$.  We must find $Y \in R_S^n$ such that $|q(x-Y)|_S \leq E$.  
Writing $x = \frac{a}{b}$ with $a \in R^n$ and $b \in R^{\bullet}$ and clearing 
denominators, it suffices to find $y_S \in R_S^n$ such that 
\[ |q(a-by_S)|_S \leq E |b|_S^2. \]
As above, we may factor $b$ as $s_b b'$ with $s_b \in S$ and $b'$ prime to $S$, 
so $|b'|_S = |b'|$.  Applying our hypothesis to the element $\frac{a}{b'}$ of $K^n$ we may choose $y \in R^n$ such that 
$|q(a-b'y)| \leq E |b'|^2$.  Now put $y_S = \frac{y}{s_b}$, so 
\[|q(a-by_S)|_S = |q(a-b'y)|_S \leq |q(a-b'y)| \leq E |b'|^2 = E |b'|_S^2 = E|b|_S^2. \] 
\end{proof}

\begin{cor}
Retain the notation of Theorem \ref{UFDLOCTHM} and write $q_S$ for $q$ viewed as a quadratic form on the normed ring $(R_S,| \cdot |_S)$.  Then: \\
a) $E(q_S) \leq E(q)$.  \\
b) If $q$ is Euclidean, so is $q_S$.
\end{cor} 
\begin{proof} a) By definition of the Euclideanity, for all $\epsilon > 0$ 
and all $x \in K^n$, there exists $y \in R^n$ such that $|q(x-y)| \leq E(q) + \epsilon$.  Therefore  Theorem \ref{UFDLOCTHM} applies with $E = E(q) + \epsilon$ to show that for all $x \in K$, there exists $y_S \in R_S$ 
with $|q(x-y_S)|_S \leq E(q) + \epsilon$, i.e., $E(q_S) \leq E(q) + \epsilon$.  
Since $\epsilon$ was arbitrary, we conclude $E(q_S) \leq E(q)$.  \\
b) If in the statement of Theorem \ref{UFDLOCTHM} we take $E = 1$ and replace 
all the inequalities with strict inequalities, the proof goes through verbatim.
\end{proof}    
\noindent
The rings of most interest to us are Hasse domains, which of course need not 
be UFDs but are always Dedekind domains.  Thus it will be useful to have Dedekind domain analogues of the previous discussion.
\\ \\
Let $R$ be a Dedekind domain endowed with an ideal norm $|\cdot |$.  Let $R'$ be 
an \textbf{overring} of $R$, i.e., a ring intermediate between $R$ and its 
fraction field $K$: let $\iota: R \hookrightarrow S$ be the inclusion map.  Then the induced map on spectra $\iota^*: \Spec R' \ra \Spec R$ is also an injection, 
and $S$ is completely determined by the image $W := \iota^*(\Spec R')$.  Namely \cite[Cor. 6.12]{Larsen-McCarthy}  \[R' = R_W := \bigcap_{\mathfrak{p} \in W} R_{\mathfrak{p}}. \]
This allows us to identify the monoid $\mathcal{I}(R_W)$ of ideals of $R_W$ as the free submonoid of the free monoid $\mathcal{I}(R)$ on the subset $W$ of $\Spec R$ and thus define an \textbf{overring ideal norm} $| \cdot |_W$ 
on $R_W$ as the composite map 
$\mathcal{I}(R_W) \ra \mathcal{I}(R) \stackrel{| \cdot |}{\ra} \Z^+$.  
\\ \\
Remark 3.1.2.: As above, we single out the following properties of $| \cdot |_W$: \\ \\
$\bullet$ Every ideal $I \in \mathcal{R}$ may be uniquely decomposed as 
$W_I I'$ where $W_I$ is divisible by the primes of $W$ and $I'$ is prime to 
$W$, and we have

\[ |I|_W = |W_I I'|_S = |I'|_S = |I'|. \]
$\bullet$ For all ideals $I$, $|I|_W \leq |I|$.  

\begin{thm}
\label{DEDLOCTHM}
Let $R$ be a Dedekind domain with fraction field $K$, $| \cdot |$  an ideal norm 
on $R$, $W \subset \Sigma_R$ and $R_W = \bigcap_{\mathfrak{p} \in W} R_{\mathfrak{p}}$ the corresponding overring.  Let $q(x) \in R[x]$ be a 
quadratic form, and suppose that $E \in \R^{> 0}$ is a constant such that 
for all $x \in K^n$, there exists $y \in R^n$ such that $|q(x-y)| \leq E$.  
Then for all $x \in K^n$, there exists $y_W \in R_W^n$ such that 
$|q(x-y_W)|_W \leq E$.
\end{thm}
\begin{proof} The argument is similar to that of Theorem \ref{UFDLOCTHM}.  The only point which requires additional attention is the existence of a 
decomposition of $b \in R^{\bullet}$ as $b = w_b b'$ with $w_b$ divisible only 
prime ideals in $W$ and $b'$ prime to $W$.  But this follows by
weak approximation (or the Chinese Remainder Theorem) applied to the finite set of prime ideals $\mathfrak{p} \in W$ which appear in the 
prime factorization of $(b)$.  
\end{proof}
\noindent
Also as before, we deduce the following result.

\begin{cor}
Retain the notation of Theorem \ref{DEDLOCTHM} and write $q_W$ for $q$ viewed as a quadratic form on the ideal normed ring $(R_W,| \cdot |_W)$.  Then: \\
a) $E(q_W) \leq E(q)$.  \\
b) If $q$ is Euclidean, so is $q_W$.
\end{cor} 

\subsection{Localization and Completion of ADC-forms}

\begin{thm}
\label{ADCLOCTHM}
Let $R$ be a domain, $S \subset R^{\bullet}$ a saturated multiplicatively closed subset and $R_S = S^{-1} R$ the 
localized domain.  If a quadratic form $q(x) \in R[x]$ is ADC, then $q$ viewed as a quadratic form over 
$R_S$ is ADC.
\end{thm}
\begin{proof}
Let $d \in R_S^{\bullet}$ be $K$-represented by $q_S$, i.e., there exists $x \in K^n$ such that $q(x) = d$.  We may 
write $d = \frac{a}{s}$ with $s \in S$.  If $x = (x_1,\ldots,x_n)$, then by $sx$ we mean $(sx_1,\ldots,sx_n)$.  Thus 
$q(sx) = s^2 q(x) = sa \in R$.  Since $q$ is ADC over $R$, there exists $y \in R^n$ such that $q(y) = sa$.  But then 
$s^{-1} y \in R_S^n$ and $q(s^{-1} y) = \frac{a}{s}$.
\end{proof}

\begin{cor}
\label{ADCLOCCOR}
Let $R$ be a Dedekind domain with fraction field $K$, let $v: K^{\bullet} \ra \Z$ be a nontrivial discrete valuation 
which is ``$R$-regular'' in the sense that $R$ is contained in the valuation ring $v^{-1}(\N) \cup \{0\}$.  Let $K_v$ 
be the completion of $K$ with respect to $v$ and $R_v$ its valuation ring.  Suppose $q \in R[x]$ is an ADC form.  
Then the base extension of $q$ to $R_v$ is an ADC-form.
\end{cor}
\begin{proof} Under the hypotheses of the theorem, $v = v_{\mathfrak{p}}$ for a nonzero prime ideal $\mathfrak{p}$ of $R$.  Let $S = R \setminus \mathfrak{p}$, and put $R_S = S^{-1} R$.  By the previous 
theorem, the extension of $q$ to $R_S$ is an ADC form.  Now suppose $D \in R_v^{\bullet}$ is such that there exists 
$X \in K_v^n$ with $q(X) = D$.  We may choose $x \in K^n$ which is sufficiently $v$-adically close to $X$ so that 
$q(x) = d \in R_S$ and $\frac{D}{d} = u_d^2$ for some $u_d \in R_v^{\times}$.  (This is possible because: $R_S^n$ is 
dense in $R_v^n$, $q$, being a polynomial function, is continuous for the $v$-adic topology, and $R_v^{\times 2}$ is an open subgroup of $R_v^{\bullet}$: e.g. \cite[Thm 3.39]{Gerstein}.)  Since $q$ is ADC over $R_S$, there exists $y \in R_S^n$ such that $q(y) = d$.  Thus $q(u_d y) = u_d^2 d = D$, showing that $D$ is $R_v$-represented by $q$.
\end{proof}  

\section{CDVRs and Hasse Domains}

\subsection{Basic defintions} \textbf{} \\ \\ \noindent
Let $(R,v)$ be a discrete valuation ring (DVR) with fraction field $K$ and residue field $k$.  As usual, we require that the characteristic of $K$ be different from $2$; however, although it is invariably more troublesome, we certainly must admit the case in which $k$ has characteristic $2$: such DVRs are called \textbf{dyadic}.  
We will be especially interested in the case in which $R$ is complete, a \textbf{CDVR}.
\\ \\
A \textbf{Hasse domain} is the ring of $S$-integers in a number field $K$ or the coordinate ring of a regular, integral algebraic curve 
over a finite field $k = \mathbb{F}_q$.  (The terminology is taken from \cite{OMeara}.)  In particular, a Hasse domain is a Dedekind finite quotient 
domain.
\\ \\
Let $\Sigma_K$ denote the set of all places of $K$, including Archimedean ones in the number field case.  Let $\Sigma_R = \Sigma_K \setminus S$ denote the subset of $\Sigma_K$ consisting of places which correspond to maximal ideals of $R$; these places will be called \emph{finite}.  The completion $R_v$ of a Hasse domain $R$ at $v \in \Sigma_R$ is a CDVR with finite residue field.  
\\ \\
If $R$ is a Hasse domain and $\Lambda$ is a quadratic $R$-lattice in the quadratic space $(V,q)$, then to each $v \in \Sigma_R$ we may attach 
the local lattice $\Lambda_v = \Lambda \otimes_R R_v$.  Being a finitely 
generated torsion-free module over the PID $R_v$, $\Lambda_v$ is necesssarily 
free.  In particular, we may define $\delta_v$, the valuation of the discriminant over $R_v$ and then the global discriminant may be defined as 
the ideal $\Delta(\Lambda) = \prod_{v \in \Sigma_R} \mathfrak{p}_v^{\delta_v}$. 
\begin{lemma}
\label{EASYMAXLEMMA}
\textbf{} \\ 
a) The $R$-lattice $\Lambda$ is maximal iff $\Lambda_v$ is a maximal $R_v$-lattice for all $v \in \Sigma_R$.  \\
b) For any \emph{nondyadic} place $v$ such that $\delta_v(\Lambda) \leq 1$, 
the lattice $\Lambda_v$ is $R_v$-maximal.
\end{lemma}
\begin{proof} For part a), see \cite[$\S$ 82K]{OMeara}.  For part b), see 
\cite[82:19]{OMeara}.
\end{proof}

\subsection{Classification of Euclidean forms over CDVRs} \textbf{} \\ \\ \noindent
In this section $R$ is a CDVR with fraction field $K$ of characteristic different from $2$, endowed with the norm $| \cdot |_a$ (for some $a \geq 2$) of Example 1.3.  In this 
setting we can give a very clean characterization of Euclidean forms.

\begin{thm}
\label{EUCEQUALSMAXTHM}
A quadratic form over a complete discrete valuation domain is Euclidean for the 
canonical norm iff the corresponding quadratic lattice is maximal.
\end{thm}
\noindent
For the proof we require the following preliminary results.

\begin{thm}(Eichler's Maximal Lattice Theorem) Let $q$ be an anisotropic quadratic form over a complete discrete valuation field $K$ with valuation ring $R$.  Then there is a unique maximal $R$-lattice for $q$, namely 
\[ \Lambda = \{x \in K^n \ | \ q(x) \in R \} . \]
\end{thm}
\begin{proof}
See \cite{Eichler} or \cite[Thm. 8.8]{Gerstein}.
\end{proof}

\begin{thm}
\label{SHIMURA29.8}
Let $(V,q)$ be a finite-dimensional quadratic space over $K$ and 
$\Lambda \subset V$ a maximal quadratic $R$-lattice.  Then there exists 
a decomposition 
\[V = \bigoplus_{i=1}^r {\mathbb{H}}_K \oplus V' \]
with $q|_{V'}$ anisotropic such that 
\[ \Lambda = \bigoplus_{i=1}^r {\mathbb{H}}_R \oplus \Lambda ', \]
where $\Lambda ' = \Lambda \cap V'$.
\end{thm}
\begin{proof} See \cite[Lemma 29.8]{Shimura}, wherein the result is stated for 
complete discrete valuation rings with finite residue field.  However, it is easy to see that the finiteness of the residue field is not used in the proof.
\end{proof}
\noindent
\emph{Proof of Theorem \ref{EUCEQUALSMAXTHM}}: By Proposition \ref{EUCIMPLIESMAXPROP}, it is enough to show that any maximal $q_{/R}$ 
is Euclidean. \\ \indent
Suppose first that $q$ is anisotropic over $R$.  In this case, the Euclideanness of $q$ follows immediately from Eichler's Maximal Lattice Theorem: indeed, we have 
\[R^n = \{x \in K^n \ | \ |q(x)|_a \geq 1 \}. \]  
Therefore, 
$x \in K^n \setminus R^n \iff |q(x)|_a = |q(x-0)|_a < 1$. \\
We now deal with the general case.  By Theorem \ref{SHIMURA29.8}, we may write $\Lambda = \bigoplus_{i=1}^r \mathbb{H}_R \oplus \Lambda'$ with $\Lambda'$ anisotropic.  With respect to a suitable $R$-basis of $\Lambda$, $q$ takes the form
\[q(X) = q(x,x') = x_1 x_2 + \ldots + x_{2r-1} x_{2r} + q'(x'), \]
where $x' = (x_{2r+1},\ldots,x_{n})$ and $q'$ is anisotropic.  Let $X = (x,x') \in K^n \setminus R^n$.  We must find $Y = (y,y') \in R^n$ such that $v(q(X-Y))) < 0$.  By symmetry, we may assume that $v(x_1 x_2) \geq \ldots \geq v(x_{2r-1} x_{2r})$ and $v(x_{2r}) \leq 
v(x_{2r-1})$.  \\
Case 1: $v(x_{2r}) \geq 0$.  Then $x = (x_1,\ldots,x_{2r}) \in R^{2r}$ so that 
we must have $x' \in K^{n-2r} \setminus R^{n-2r}$.  Put $Y = (y,y') = 0$.  Then 
$v(x_1 x_2 + \ldots + x_{2r-1} x_{2r}) \geq 0$, whereas by Eichler's Maximal Lattice Theorem, $v(q'(x')) < 0$, so \[v(q(X)) = v(x_1 x_2 + \ldots + x_{2r-1} x_{2r} + q'(x')) < 0. \]  
Case 2: $v(x_{2r}) < 0$.  We choose $y' = 0$ and $y_1 = \ldots = y_{2r-2} = 0$.  Also define 
\[\alpha = q_2(x'), \ \beta = x_1 x_2 + \ldots + x_{2r-3} x_{2r-2}. \]
If $v(\alpha + \beta + x_{2r-1} x_{2r}) \leq v(x_{2r})$, then since 
$v(x_{2r}) < 0$, we may take $y = 0$, getting 
\[v(q(X)) = v(\alpha + \beta + x_{2r-1} x_{2r}) < 0. \] 
If $v(\alpha + \beta + x_{2r-1} x_{2r-2}) > v(x_{2r})$, we may take $y_{2r-1} = 1$, $y_{2r} = 0$, getting \[v(q(X-Y)) = v(\alpha + \beta + x_{2r-1} x_{2r} - x_{2r}) = v(x_{2r}) < 0. \]

\begin{cor}
\label{LOCALLYMAXCOR}
Let $R$ be a Hasse domain and $q_{/R}$ a quadratic form.  Then $q$ is locally 
Euclidean iff the corresponding lattice $\Lambda_q$ is maximal.
\end{cor}
\begin{proof} This is an immediate consequence of Theorem \ref{EUCEQUALSMAXTHM} and Lemma \ref{EASYMAXLEMMA}.
\end{proof}

\subsection{ADC forms over Hasse domains}
\textbf{} \\ \\ \noindent
Let $q_{/R}$ be a nondegenerate quadratic form.  We define the \textbf{genus} $\mathfrak{g}(q)$ as follows: it is 
the set of $R$-isomorphism classes of quadratic forms $q'$ such that: for each $v \in S$, $q \cong_{K_v} q'$, and for 
each $v \in \Sigma_R$, $q \cong_{R_v} q'$.  

\begin{thm}
\label{SETUP1}
For any nondegenerate quadratic form $q$ over a Hasse domain $R$, the genus $\mathfrak{g}(q)$ of $q$ is finite.
\end{thm}
\begin{proof} \cite[Thm. 103:4]{OMeara}.
\end{proof} 
\noindent
This allows us to define the \textbf{class number} $h(q)$ of a quadratic form $q$ as $\# \mathfrak{g}(q)$.  Of particular 
interest are forms of class number one, i.e., for which $q$ is (up to isomorphism) the only form in its genus.
\\ \\
A quadratic form $q_{/R}$ is \textbf{regular} if it $R$-represents every element of $R$ which is represented by its genus.  
In other words, $q$ is regular if for all $d \in R$, if there is $q' \in \mathfrak{g}(q)$ and $x \in R^n$ such that 
$q'(x) = d$, then there is $y \in R^n$ such that $q(y) = d$.  

\begin{thm}
\label{SETUP2}
Let $q_{/R}$ be a nondegenerate quadratic form over a Hasse domain, and let $d \in R$.  Suppose that for all $v \in S$, 
$q$ $K_v$-represents $d$ and for all $v \in \Sigma_R$, $q$ $R_v$-represents $d$.  Then there exists $q' \in \mathfrak{g}(q)$ such that $q'$ $R$-represents $d$.%
\end{thm}
\begin{proof} \cite[102:5]{OMeara}.
\end{proof}

\begin{thm}
\label{SETUP4}
For a form $q$ over a Hasse domain $R$, the following are equivalent: \\
(i) $q$ is an ADC form. \\
(ii) $q$ is regular and ``locally ADC'': for all $\mathfrak{p} \in \Sigma(R)$, $q$ is ADC over $R_{\mathfrak{p}}$.
\end{thm}
\begin{proof}
(i) $\implies$ (ii): Suppose $q$ is ADC.  By our theorems on localization, $q$ is locally ADC.  Now let $d \in R$ be represented by the genus of $q$: i.e., there exists $q' \in \mathfrak{g}(q)$ such that $q'$ $R$-represents $d$.  Since for all $v \in \Sigma_K$, $q' \cong_{K_v} q$, it follows that $q$ $K_v$-represents $d$ for all $v$.  By Hasse-Minkowski, $q$ $K$-represents $d$, and since $q$ is an ADC-form, $q$ $R$-represents $d$.  \\
(ii) $\implies$ (i): Suppose $q$ is regular and locally ADC, and let $d \in R$ be $K$-rationally represented by $q$.  
Then for all $v \in \Sigma(R)$, $d$ is $K_v$-represented by $q$, hence using the local 
ADC hypothesis, is $R_v$-represented.  Moreover, for all places $v \in \Sigma(K) \setminus \Sigma(R)$, $d$ is $K_v$-represented by $q$.  By Theorem \ref{SETUP2}, there exists $q' \in \mathfrak{g}(q)$ which $R$-represents $d$, 
and then by definition of regular, $q$ $R$-represents $d$.
\end{proof}    
\noindent
A quadratic form $q$ over a Hasse domain $R$ is \textbf{sign-universal} if for all $d \in R$, if $q$ $K_v$-represents 
$d$ for all real places $v \in \Sigma_K$, then $q$ $R$-represents $d$.

\begin{prop}
\label{SETUP5}
Let $n \geq 4$, and let $q(x_1,\ldots,x_n)$ be a nondegenerate quadratic form over a Hasse domain $R$.  Then $q$ is ADC iff it is sign-universal.
\end{prop}
\begin{proof} Indeed, by the Hasse-Minkowski theory of quadratic forms over global fields, any nondegenerate quadratic 
form in at least four variables over the fraction field $K$ is sign-universal.  The result follows immediately from this.
\end{proof}  
\noindent

\subsection{Conjectures on Euclidean Forms over Hasse Domains}

\begin{conj}
\label{FINITENESSCONJ}
For any Hasse domain $R$, there are only finitely many isomorphism classes 
of anisotropic Euclidean forms $q_{/R}$.
\end{conj}

\begin{conj}
\label{CLASSNUMBERONECONJ}
\label{MAINCONJ}
Let $q$ be an anisotropic Euclidean quadratic form over a Hasse domain $R$.  
Then $q$ has class number one.
\end{conj}
\noindent
Conjecture \ref{CLASSNUMBERONECONJ} has a striking consequence.  Consider 
the set $\mathcal{S}_1$ of all class number one totally definite quadratic forms 
defined over the ring of integers of some totally real number field.    Work of Siegel shows that $\mathcal{S}_1$ is a finite set.  Thus Conjecture \ref{CLASSNUMBERONECONJ} implies the following 
result, which we also state as a conjecture.  

\begin{conj}
\label{TOTALLYREALCONJ}
As $R$ ranges through all rings of integers of totally real number fields, 
there are only finitely many totally definite Euclidean quadratic forms 
$q_{/R}$.
\end{conj}

\subsection{Definite Euclidean forms over $\Z$} \textbf{} \\ \\ \noindent
In the case of $R = \Z$, Conjecture \ref{FINITENESSCONJ} is intimately related to fundamental problems in the geometry of numbers.  Especially, 
the classification of definite Euclidean forms $q_{/\Z}$ can be rephrased 
as the classification of all integral lattices in Euclidean space with 
\textbf{covering radius} strictly less than $1$. \\ \indent  
This problem has been solved by G. Nebe \cite{Nebe}, subject to the following proviso.  Nebe's paper contains $69$ Euclidean lattices.  Before becoming 
aware of \cite{Nebe} W.C. Jagy and I had been independently searching 
for Euclidean lattices.  Our search was not exhaustive, i.e., we looked for 
and found Euclidean lattices in various places but without any claim of finding 
all of them.  When we learned of Nebe's work we compared out list to hers and 
found that her list contained several lattices that we did not have.  However, one of our lattices does not appear on Nebe's list, 
\[q(x_1,x_2,x_3,x_4,x_5) = x_1^2 + x_1 x_4 + x_2^2 + x_2 x_5 + x_3^2 + x_3x_5 + x_4^2 + x_4x_5 + 2x_5^2. \]
We contacted Professor Nebe and she informed us that this lattice was not 
included due to a simple oversight in her casewise analysis.  So we get the 
following result.
\begin{thm}(Nebe)
\label{NEBETHM}
There are precisely $70$ positive definite Euclidean quadratic forms over $\Z$.  
All of these lattices have class number one.
\end{thm}
\noindent
The second sentence in Theorem \ref{NEBETHM} follows easily by explicit 
computation, for instance using the command {\tt GenusRepresentatives} in the MAGMA software package.  Thus Theorem \ref{NEBETHM} verifies Conjecture 
\ref{CLASSNUMBERONECONJ} for definite forms over $\Z$.  

\subsection{Definite ADC forms over $\Z$} 
\textbf{} \\ \\
The work of this paper allows us to classify (in a certain sense) primitive definite ADC forms over $\Z$.  Indeed, by Theorem \ref{SETUP4}, it suffices 
to classify the regular primitive positive definite forms over $\Z$ and 
for each such form $q$ determine whether it is locally ADC.  The theory of 
quadratic forms over $p$-adic integer rings is completely understood, to the 
extent that for a fixed quadratic form $q_{/\Z}$, determining for all 
primes $p$ the set of all elements of $\Z_p$ (resp. $\Q_p$) which are 
$\Z_p$-represented (resp. $\Q_p$-represented) by $q$ is a finite problem.  
So if we could reduce ourselves to a finite set of regular forms, the problem 
would be solved modulo a finite calculation.  Let us see how this procedure 
works out for forms in various dimensions.
\\ \\
{\sc Unary forms}: Let $a \in \Z^{\bullet}$.  Recall Example 2.5: a unary form $q_a(x) = ax^2$ is ADC iff $a$ is squarefree.  \\ \indent 
In fact we have shown that for any UFD or Dedekind domain $R$ and $a \in R^{\bullet}$, the 
unary form $q_a(x) = ax^2$ is ADC iff $\ord_{\pp}(a) \leq 1$ for every 
height one prime ideal $\pp$ of $R$.  But it seems premature to present such results here, since this is an easy special case of a not so easy general problem.  Let us say a form $q(x)$ is \textbf{imprimitive} 
if it can be written as $a q'(x)$ with $a \in R^{\bullet} \setminus R^{\times}$.  
Then we would like to know: if $q'(x)$ is a primitive ADC form, for which 
$a \in R^{\bullet}$ is $a q'(x)$ an ADC form?  We can answer this for unary 
forms but not in general.  We leave the general problem of imprimitive 
forms for a later work.  
\\ \indent
So up to unit equivalence the unique \emph{primitive} ADC unary form over $\Z$ 
is $x^2$.
\\ \\
{\sc Binary forms}: The classical genus theory shows that a regular binary 
form $q(x,y) = ax^2 + bxy + cy^2$ has class number one in the above sense.  There is however a subtlety here in that classes and genera of binary 
quadratic forms $q(x,y)_{/\Z}$ are classically expressed in terms of \textbf{proper equivalence} (i.e., $\SL_2(\Z)$-equivalence).  To get from 
the proper genera to the genera one needs to identify each class with its inverse in the class group: we get a quotient map which has fibers of cardinality one over the order two elements of the class group and cardinality $2$ otherwise.  Thus, in addition to the binary quadratic forms which have proper (form) class number one -- i.e., the \textbf{idoneal} discriminants 
$\Delta = b^2-4ac$ such that the quadratic order of disciminant $\Delta$ 
has $2$-torsion class group -- we need to consider \textbf{bi-idoneal} forms in the sense of \cite{Jagy-Kaplansky} and \cite{Voight}, i.e., forms of order $4$ 
in a class group of type $\Z/4\Z \times (\Z/2\Z)^a$ for $a \geq 0$.  (C.f. Remarks 2.5, 2.6 and 4.6 of \cite{Voight} for a clear explanation of the relationship between binary forms of $\operatorname{GL}_2(\Z)$-genus one and class groups of the above form.)  Voight computes a list of $425$ bi-idoneal discriminants, shows that this list is complete except for possibly one further (very large) 
value, and shows that the Generalized Riemann Hypothesis (GRH) implies the completeness of his list.  These results allow us to give a complete enumeration of primitive binary definite ADC forms over $\Z$, conditionally on GRH.
\\ \\
Again the issue of imprimitive forms requires some additional 
consideration.\footnote{Added, October 2011: we can now handle the imprimitive forms as well.}
\\ \\
Example 4.1: Let $q' = x^2 + y^2$.  Then $q'$ is Euclidean hence ADC.  The form $a q'$ is squarefree 
iff $a$ is squarefree and not divisible by any prime $p \equiv 1 \pmod 4$.  
\\ \\
{\sc Ternary forms:}

\begin{thm}(Jagy-Kaplansky-Schiemann \cite{JKS}) There are at most $913$ primitive positive definite regular forms $q(x_1,x_2,x_3)_{/\Z}$.
\end{thm}
\noindent
More precisely, in \cite{JKS} the authors write down an explicit list of $913$ 
definite ternary forms such that any regular form must be equivalent to some 
form in their list.  Further they prove regularity of $891$ of the forms 
in their list, whereas the regularity of the remaining $22$ forms is conjectured 
but not proven.  
\\ \indent
Fortunately, all $22$ of the forms whose regularity was not shown in \cite{JKS} 
turn out \emph{not} to be ADC-forms.  To show this one need only 
supply a \textbf{non-ADC certificate}, i.e., a pair $(a,b) \in \Z^2$ 
such that $q$ $\Z$-represents $a^2 b$ but not $b$.  Jagy has found non-ADC 
certificates for all $22$ of the possibly nonregular ternary forms above 
and indeed for the majority of the $913$ regular forms as well: his 
computations leave a list of $104$ primitive definite ternary regular forms which are probably ADC.  As above, we are left with a (nontrivial) finite 
local calculation to confirm or deny the ADC-ness of each of these $104$ forms. 
\\ \\
{\sc Quaternary Forms:} By Proposition \ref{SETUP5}, a quadratic form $q_{/\Z}$ in at least four variables is ADC iff it is sign-universal.  Thus the following result solves the problem for us when $n = 4$.
\begin{thm}(Bhargava-Hanke \cite{Bhargava-Hanke})
\label{BHTHM1} There are precisely $6436$ positive definite sign-universal forms $q(x_1,x_2,x_3,x_4)_{/\Z}$.
\end{thm}
\noindent
So there are precisely $6436$ positive definite quaternary ADC forms over $\Z$.
\\ \\
{\sc Beyond Quaternary Forms:} It seems hopeless to classify positive definite sign-universal forms in $5$ or more variables.  In contrast to all cases above, there are most 
certainly \emph{infinitely many} such primitive forms, e.g. $x_1^2 + \ldots + x_{n-1}^2 + D x_n^2$.  More generally, any form with a sign-universal subform is obviously sign-universal, and this makes the problem difficult.  However, there is the 
following relevant result.   
\begin{thm}(Bhargava-Hanke \cite{Bhargava-Hanke})
\label{BHTHM2}
A positive definite form $q(x_1,\ldots,x_n)_{/\Z}$ is sign-universal if and only if it integrally represents the first $290$ positive integers.
\end{thm}
\noindent
Thus a positive definite integral form $q(x_1,\ldots,x_n)$, $n \geq 4$, is ADC iff it represents the integers listed in Theorem \ref{BHTHM2}.  This gives a kind of classification for definite ADC forms in at least five variables, and 
one can probably do no better than this.

\subsection{Definite ADC forms over $\F[t]$}
\textbf{} \\ \\ \noindent
Let $\F$ be a finite field of odd order, $\delta \in \F^{\times} \setminus \F^{\times 2}$, $R = \F[t]$ be endowed with its canonical norm, $K = \F(t)$, and $\infty$ be the infinite place of $K$ (corresponding to the valuation $v_{\infty}(\frac{f}{g}) = \deg(g) - \deg(f)$), so that $K_{\infty} = K((\frac{1}{t}))$.  
\\ \\
Recall that $K$ has $u$-invariant $4$: i.e., the maximum dimension of an anisotropic quadratic form over $R$ is $4$.  We call a quadratic form $q_{/R}$ 
\textbf{definite} if $q$ is anisotropic as a quadratic form over $K_{\infty}$: 
in particular, such forms are aniostropic.   
\\ \\
Thus we we get a problem analogous to the $R = \Z$ case: find all 
definite forms over $\F[t]$ which are Euclidean and which are ADC forms.  There 
are however some significant differences from the $R = \Z$ case.  We saw one 
above: we can \emph{a priori} restrict to forms of dimension at most $4$.  Here 
is another striking difference.
\begin{thm}(Bureau \cite{Bureau})
Suppose that $\# \F > 3$.  Then every regular definite form $q_{/\F[t]}$ has 
class number one.
\end{thm}
\noindent
In particular -- excepting $\F = \F_3$ -- we have Euclidean implies ADC 
implies regular implies class number one -- so Conjecture \ref{MAINCONJ} 
holds for definite Euclidean forms over $\F[t]$.  Moreover, there are only finitely many definite quadratic forms over 
$\F[t]$ of any given class number, so this verifies Conjecture 
\ref{FINITENESSCONJ} for definite forms over $R$.  
\\ \\
We end with a few preliminary results towards the classification of Euclidean 
and ADC forms over $\F[t]$, mostly to showcase the connection to 
Theorem \ref{GENCPTHM}. 

\begin{thm} For a definite quaternary form $q_{/\F[t]}$, the following are equivalent: \\
(i) $q$ is ADC. \\
(ii) $q$ is universal. \\
(iii) The discriminant of $q$ has degree $2$.
\end{thm}
\begin{proof}
(i) $\iff$ (ii) is a case of Proposition \ref{SETUP5}. \\
(ii) $\iff$ (iii): this is a result of W.K. Chan and J. Daniels \cite[Cor. 4.3]{Chan-Daniels}.
\end{proof}

\begin{thm}
For a \emph{diagonal} definite quaternary form $q$ over $\F[t]$, the following are equivalent: \\
(i) $q$ is Euclidean.  \\
(ii) $q$ is universal. \\
(iii) The discriminant of $q$ has degree $2$.
\end{thm}
\begin{proof} \textbf{} \\ \noindent
(i) $\implies$ (ii) follows from Theorem \ref{MAINTHM} and Proposition \ref{SETUP5}. \\
(ii) $\implies$ (iii) is immediate from the previous result.  \\
(iii) $\implies$ (i): Suppose \[q = p_1 x_1^2 + p_2 x_2^2 + p_3 x_3^2 + p_4 x_4^2\] Without loss of generality, we may assume that 
$\deg(p_1) \leq \deg(p_2) \leq \deg(p_3) \leq deg(p_4)$.  If $\deg(p_3) = 0$, 
then $q$ contains a $3$-dimensional constant subform and is thus isotropic.  
Since $\sum_i \deg(p_i) = 2$, the only other possibility is $\deg(p_1) = \deg(p_2) = 0$, $\deg(p_3) = \deg(p_4) = 1$, and now the fact that $q$ is Euclidean follows from the Generalized Cassels-Pfister Theorem.
\end{proof}

\begin{thm}
If $q$ is a \emph{diagonal} definite ternary form over $\F[t]$ with 
$\deg(\Delta(q)) \leq 2$, then $q$ is ADC.
\end{thm}
\begin{proof}
By \cite[Thm. 3.5]{Chan-Daniels} any definite ternary form over $\F[t]$ 
with $\deg(\Delta(q)) \leq 2$ has class number one, hence is regular.  Therefore, by Theorem \ref{SETUP4} it is sufficient to show that $q$ is 
locally ADC. \\ \indent
If $\deg(\Delta(q)) \leq 1$, then since $R$ is nondyadic, the corresponding 
lattice is maximal, hence locally ADC by Theorem \ref{SETUP4} and Corollary \ref{LOCALLYMAXCOR}. \\ \indent
Suppose $\deg(\Delta(q)) = 2$ and write $q = p_1(t)x_1^2 + p_2(t) x_2^2 + 
p_3(t) x_3^2$ with $\deg(p_1) \leq \deg(p_2) \leq \deg(p_3)$.  If $\deg(p_3) = 1$, then by the Generalized Cassels-Pfister Theorem $q$ is Euclidean.  Otherwise 
$\deg(p_1) = \deg(p_2)= 0$ and $\deg(p_3) = 2$.  If $p_3$ is squarefree then 
so is $\Delta(q)$, hence $q$ is maximal and thus locally ADC.  Otherwise 
there exist $a \in \F^{\times}, \ b \in \F$ such that $p_3 = a(t-b)^2$, but 
then $q$ is equivalent over $K$ to the constant form $p_1 x_1^2 +p_2 x_2^2 + a x_3^2$ and is therefore isotropic, a contradiction.
\end{proof}
\noindent
Again, a complete classification -- over any fixed finite field $\F$ -- 
is reduced to a finite calculation.  We hope to give precise classification 
theorems in a future work.

\end{document}